\documentclass[10pt]{amsart}
\usepackage{latexsym,amsmath,amssymb}
\usepackage{mathrsfs}
\usepackage{mathabx}
 \usepackage{indentfirst}
 \usepackage{galois}
 \usepackage{color}
 \renewcommand{\epsilon}{\varepsilon}
 \newcommand{\newsection}[1]
  {\section{#1}\setcounter{theorem}{0} \setcounter{equation}{0}\par\noindent}

   \newtheorem{theorem}{Theorem}[section]
   \newtheorem{lemma}[theorem]{Lemma}
 \newtheorem{corr}[theorem]{Corollary}
 
 \newtheorem{proposition}[theorem]{Proposition}
 \newtheorem{deff}[theorem]{Definition}
 \newtheorem{remark}[theorem]{Remark}

  \numberwithin{equation}{section}

 \newcommand{\bth}{\begin{theorem}}
 \newcommand{\ble}{\begin{lemma}}
 \newcommand{\bcor}{\begin{corr}}
 \newcommand{\bdeff}{\begin{deff}}
 \newcommand{\bprop}{\begin{proposition}}
 \def\be{\begin{equation}}
\def\ee{\end{equation}}
\def\bt{\begin{theorem}}
\def\et{\end{theorem}}
\def\ba{\begin{array}}
\def\ea{\end{array}}
\def\bl{\begin{lemma}}
\def\el{\end{lemma}}

 \newcommand{\ele}{\end{lemma}}
 \newcommand{\ecor}{\end{corr}}
 \newcommand{\edeff}{\end{deff}}
 
 \newcommand{\eprop}{\end{proposition}}

 \renewcommand{\Pi}{\varPi}

 \renewcommand{\epsilon}{\varepsilon}

\title[The Einstein-Klein-Gordon system] {Global regularity for Einstein-Klein-Gordon system with $U(1) \times \mathbb{R}$ isometry group, \uppercase\expandafter{\romannumeral1}}

\date{\today}

\begin{document}
\maketitle

\centerline{
\author{Haoyang Chen
  \footnote{School of Mathematical Sciences, Fudan University, Shanghai, China. {\it Email: 15110180018@fudan.edu.cn }}
 }
 \and
  Yi Zhou
  \footnote{ *Corresponding Author: School of Mathematical Sciences, Fudan University, Shanghai, China.
 {\it Email: yizhou@fudan.edu.cn}
 }
  }

\begin{abstract}
This is the first of the two papers devoted to the study of global regularity of the 3+1 dimensional Einstein-Klein-Gordon system with a $U(1)\times \mathbb{R}$ isometry group. In this first part, we reduce the Cauchy problem of the Einstein-Klein-Gordon system to a 2+1 dimensional system. Then, we will give energy estimates and construct the null coordinate system, under which we finally show that the first possible singularity can only occur at the axis.
\end{abstract}


\newsection{Introduction}
\subsection{Introduction and previous results}
Let $(^{(4)}{M},^{(4)}{g})$ be a 3+1 dimensional globally hyperbolic Lorentzian manifold which satisfies the following Einstein-scalar field equations:
\begin{equation}\label{1.1}
\begin{cases}
^{(4)}{G}_{\mu \nu}=^{(4)}{T}_{\mu\nu}=\partial_\mu \phi \partial_\nu \phi-\frac{1}{2} {^{(4)}{g}}_{\mu \nu}\partial^{\lambda} \phi \partial_{\lambda} \phi-{^{(4)}{g}}_{\mu\nu}V(\phi) \\
\Box_{^{(4)}{g}} \phi=V'(\phi)
\end{cases}
\end{equation}
where $^{(4)}{g}$ is the Lorentzian metric, $^{(4)}{G}_{\mu \nu}$ is the Einstein tensor of $^{(4)}{g}$, $^{(4)}{T}_{\mu\nu}$ is the stress-energy tensor given as above, $\phi$ is the scalar field, and we take the potential $V(\phi)=\frac{1}{2}m^2 \phi^2$.

Then, the equation that the scalar field satisfies is a Klein-Gordon equation, which makes the system an Einstein-Klein-Gordon system. The Einstein-Klein-Gordon system \eqref{1.1}  is equivalent to the following equations,
\begin{equation}\label{1.2}
\begin{cases}
^{(4)}{R}_{\mu \nu}=^{(4)}{\rho}_{\mu\nu} \triangleq \partial_\mu \phi \partial_\nu \phi+{^{(4)}{g}}_{\mu\nu} \frac{1}{2}m^2 \phi^2\\
\Box_{^{(4)}{g}} \phi=m^2 \phi
\end{cases}
\end{equation}
where $^{(4)}{R}_{\mu \nu}$ is the Ricci curvature tensor.

One fundamental open problem in the field of general relativity is the cosmic censorship conjectures by Penrose. Roughly speaking, the weak cosmic censorship may be formulated as follows: For generic asymptotically flat Cauchy data, of the vacuum equations or suitable Einstein-matter systems, the maximal development possesses a complete future null infinity. While the strong cosmic censorship states that the maximal Cauchy development is inextendible for generic initial data. This question is partially related to the study of the formation of event horizons. Although still open in general, there is a series of results in spherical symmetric case by Christodoulou for the Einstein equations coupled with a massless scalar field. On the other hand, of our interest, it is related to the global well-posedness of the Cauchy problem in large.

We aim to show the  global regularity of the above 3+1 dimensional Einstein-Klein-Gordon systems with a $U(1)\times \mathbb{R}$ isometry group, while in this paper we will firstly give several results in preparation for the subsequent study on the global well-posedness.

 Our study is motivated by research on the vacuum Einstein equations with spacelike Killing vector fields which is related to the study of wave map systems. We first review some related results on the vacuum case and the wave map systems.

For the 3+1 dimensional vacuum Einstein equations with one spacelike Killing field, as we can see in \cite{choquet3} and \cite{Moncrief1}, the Einstein equations can be reduced to a 2+1 dimensional Einstein-wave map system on a 2+1 dimensional Lorentzian manifold where the target manifold is the hyperbolic space $\mathbb{H}^2$. Yet the global existence problem of the Einstein-wave map system is still open. As a first step towards this global existence conjecture, Andersson, Gudapati and Szeftel proved that the global regularity holds for the equivariant case in \cite{Andersson}, by reference to some pioneering work on equivariant wave maps.

Shatah and Tahvildar-Zadeh have proved the global regularity for 2+1 dimensional equivariant wave maps with the target geodesically convex in \cite{shatah2}. This condition was later relaxed by Grillakis to include a certain class of nonconvex targets, see \cite{grillakis}. Their proof of regularity was also simplified later by Shatah and Struwe in \cite{shatah1}. They gave several more results on equivariant wave maps in the areas of existence and uniqueness, regularity, asymptotic behavior, development of singularities, and weak solutions, see \cite{shatah3}. Further, as an improvement of these above results, for target manifolds that do not admit nonconstant harmonic spheres, global existence of smooth solutions to the Cauchy problem for corotational wave maps with smooth equivariant data was shown by Struwe in \cite{Struwe1}.

 Then, for the 3+1 dimensional vacuum Einstein equations with $G_2$ symmetry, it was shown in \cite{berger} that the system reduce to a spherically symmetric wave map $u:\, \mathbb{R}^{2+1} \to \mathbb{H}^2$, where $\mathbb{R}^{1+2}$ is the 2+1 dimensional Minkowski spacetimes and the target $\mathbb{H}^2$ is the hyperbolic space. Thus the global regularity can be proved by the work of Christodoulou and Tahvildar-Zadeh \cite{christodoulou2} on 2+1 dimensional spherically symmetric wave maps. In \cite{christodoulou2}, the range of the wave map $u$ should be contained in a convex part of the target $N$. This restriction was later shown unnecessary by Struwe in \cite{Struwe3} as the target is the standard sphere. Further, Struwe give a more general result in \cite{Struwe2}, where the target is any smooth, compact Riemannian manifold without boundary. We refer to \cite{geba} for more results and references of wave maps.

\subsection{The 3+1 dimensional spacetime with $U(1) \times \mathbb{R}$ isometry group}
In this paper, we work on the Lorentzian manifold $^{(4)}{M}=\mathbb{R} \times \mathbb{R}^2 \times \mathbb{R}$ with a Lorentzian metric $^{(4)}{g}$ on it, and we consider the polarized case where the Killing fields of $(^{(4)}{M},^{(4)}{g})$ are hypersurface orthogonal. Then, with the existence of the translational Killing vector, the metric can be written in the following form,
\begin{equation*}
^{(4)}{g}=e^{-2\gamma}{^{(3)}g}+e^{2\gamma}{(dx^3)}^2
\end{equation*}
where $\partial_{x^3}$ is the translational spacelike Killing vector field.

As we mentioned before in \cite{choquet3} and \cite{Moncrief1}, the 3+1 dimensional vacuum Einstein equations with spacelike Killing field reduce to a 2+1 dimensional Einstein-wave map system with the target manifold $\mathbb{H}^2$. We will give a similar reduction to a 2+1 dimensional Einstein-wave-Klein-Gordon system in section 2, where $\gamma$ in $^{(4)}{g}$ satisfies a wave equation and the scalar field $\phi$ satisfies a Klein-Gordon equation.

In the vacuum case with $G_2$ symmetry, the wave maps equations  reduced from the Einstein equations is a semilinear wave equations, for instance, special solutions of this case are the Einstein-Rosen waves, see \cite{cecile3} and references therein. While the major difficulty in our problem is that the wave equations are coupled with Einstein equations, which make the system quasilinear. However, we develop a new way to solve this problem in 2+1 dimension with symmetry.

Particularly, if the scalar field is massless, we can remove the condition that the Killing vector field $\partial_{x^3}$ is hypersurface orthogonal, and the metric will take the general form
\begin{equation*}
^{(4)}{g}=e^{-2\gamma}{^{(3)}g}+e^{2\gamma}{(dx^3+A_{\alpha}dx^{\alpha})}^2.
\end{equation*}
We will mention in section 2 that the system reduce to a wave map equations coupled with a linear wave equation on the Minkowski spacetimes, of which the problem left is to study the wave map system, same as in the vacuum case.

Now we assume that the reduced spacetime $(^{(3)}M,^{(3)}g)$ is a globally hyperbolic 2+1 dimensional spacetime with Cauchy surface diffeomorphic to $\mathbb{R}^2$, on which the reduced Einstein-wave-Klein-Gordon system is radially symmetric. And we assume that the $U(1)$ action on $M$ is generated by a hypersurface orthogonal Killing field $\partial_{\theta}$. In particular, we write the metric $^{(3)}g$ in the following form in this paper
\begin{eqnarray} \label{1.3}
^{(3)}g&=&-e^{2\alpha(t,r)} dt^2+e^{2\beta(t,r)}dr^2+r^2 d \theta^2\\ \nonumber
& \triangleq &\check{g}+r^2 d \theta^2
\end{eqnarray}
 where $\check{g}$ is a metric on the orbit space  $\mathcal{Q}=M/ {{\mathbb{S}}^1}$ and $r$ is the radius function, defined such that $2\pi r(p)$ is the length of the ${\mathbb{S}}^1$ orbit through $p$.

\subsection{The Cauchy data}
As we mentioned before, to study the Cauchy problem of the 3+1 dimensional Einstein-Klein-Gordon system, we can equivalently consider the Cauchy problem of the reduced 2+1 dimensional Einstein-wave-Klein-Gordon system.

Now we introduce the definition of the Cauchy data set for the 2+1 dimensional Einstein-wave-Klein-Gordon systems with $U(1)$ isometry group as follows,
\begin{deff}[Cauchy data set for the 2+1 dimensional Einstein-wave-Klein-Gordon system with $U(1)$ isometry group] \label{def1.5}
A Cauchy data set for the 2+1 dimensional Einstein-wave-Klein-Gordon system with a $U(1)$ isometry group is a 7-tuple $(M_0,g_0,K,\gamma_0,\gamma_1,\phi_0,\phi_1)$ consisting of a Remannian 2-manifold $(M_0,g_0)$ with a spacelike rotational Killing vector field and a 2-tensor $K$ which is the second fundamental form and symmetric under the same action, $\gamma_0,\gamma_1$ are initial data for the wave equation that $\gamma$ satisfies, $\phi_0,\phi_1$ are initial data for the Klein-Gordon equation that the scalar field satisfies. $g_0, K$ are functions of $r$ only and the following constraints equations hold:
\begin{equation} \label{1.4}
R_0-K_{\alpha\beta}K^{\alpha\beta}+(trK)^2=2{^{(3)}}T_{\alpha\beta} n^\alpha n^\beta
\end{equation}
\begin{equation} \label{1.5}
D^{\beta} K_{\alpha\beta}-D_{\alpha} K^{\beta}_{\beta}={^{(3)}}{T}_{\alpha\mu} n^{\mu}
\end{equation}
where $n^{\alpha}(t,r)$ is the future directed unit normal, $R_0$ is the scalar curvature on $M_0$,$D_{\alpha}$ is the intrinsic covariant derivative on $M_0$, and $^{(3)}T_{\alpha\beta}$ is the stress-energy tensor for the reduced system.
\end{deff}

For a smooth solution of the 2+1 dimensional Einstein-wave-Klein-Gordon system with $U(1)$ symmetry, it must hold that $\alpha(t,r),\beta(t,r)$ are even functions of $r$. And we give some normalisation of the metric functions $\alpha(t,r)$ and $\beta(t,r)$ on the axis. It must hold that $\beta(t,0)=0$, in order to avoid a conical singularity at the axis $\Gamma$, which means that the perimeter of a circle
of radius $r$ grows like $2\pi cr$ at the axis, instead of $2\pi r$ in the
Euclidean metric. This condition can be realized by appropriately choosing the Cauchy data such that $\beta(0,0)=0$, see section 2. Further, $\alpha(t,0)$ is determined only up to a choice of time parametrization. We shall choose a time coordinate such that $\alpha(t,0)=0.$

Finding solutions to the constraint equations is a research area in itself. Note that Cecile has proved the existence of such constraint equations in vacuum with translational Killing vector field in \cite{cecile1}\cite{cecile2}, which is used in \cite{cecile3} to prove stability in exponential time of the Minkowski spacetime in this setting. In our case, we will just briefly show that such data exists, without going further into the study of the constraints.

We will construct an asymptotically flat\footnote{We say the Cauchy data is 'asymptotically flat' in the sense of Andersson\cite{Andersson} here and hereafter.} Cauchy data set which satisfies the following initial conditions
\begin{equation} \label{1.19}
\begin{cases}
\gamma_t(0,r)=\gamma_1(r) \geq 0,\\
\gamma(0,r)=\gamma_0(r) \geq 0,\quad\gamma_r(0,r)=\gamma_0'(r)> -\frac{1}{2} r^{-1}.
\end{cases}
\end{equation}
\begin{remark} \label{remark1.1}
The following initial condition in \eqref{1.19}
\begin{equation*}
\gamma_r(0,r)=\gamma_0'(r)> -\frac{1}{2} r^{-1}
\end{equation*}
can be replaced by
\begin{equation*}
\gamma_r(0,r)=\gamma_0'(r)> -C r^{-1}
\end{equation*}
with $C$ an arbitrary positive constant. Here, we take $C=\frac{1}{2}$ just for simplicity.
\end{remark}

For metrics of the form \eqref{1.3}, it can be calculated that the second fundamental form $K$ of a Cauchy $t$-level takes the following form
\begin{displaymath}
K=K_{rr} dr^2,
\end{displaymath}
where $K_{rr}=e^{-\alpha+2\beta} \beta_t$.

For the reduced 2+1 dimensional Einstein-wave-Klein-Gordon system(see section 2), the constraint equations \eqref{1.4}\eqref{1.5} take the following form in local coordinates,
\begin{displaymath}
\begin{cases}
2e^{-2\beta} r^{-1} \beta_r=2( e^{-2\alpha} {\gamma_t}^2+e^{-2\beta}{\gamma_r}^2+\frac{1}{2}e^{-2\alpha} {\phi_t}^2+\frac{1}{2}e^{-2\beta}{\phi_r}^2+\frac{m^2}{2} e^{-2\gamma} \phi^2),\\
r^{-1}e^{-\alpha} \beta_t=2 e^{-\alpha} {\gamma_t}\gamma_r+e^{-\alpha}\phi_t{\phi_r}.
\end{cases}
\end{displaymath}

If we take $\beta_t=0$ and $\gamma_t=0$, the constraint equations become
\begin{displaymath}
\begin{cases}
2e^{-2\beta} r^{-1} \beta_r=2(e^{-2\beta}{\gamma_r}^2+\frac{1}{2}e^{-2\alpha} {\phi_t}^2+\frac{1}{2}e^{-2\beta}{\phi_r}^2+\frac{m^2}{2} e^{-2\gamma} \phi^2),\\
e^{-\alpha}\phi_t{\phi_r}=0.
\end{cases}
\end{displaymath}
\begin{remark} \label{remark1.2}
Noting that $\beta_t=0$ means that $M_0$ is a totally geodesic submanifold. Moreover, what is more interesting, a direct calculation shows that the additional condition $\gamma_t=0$ makes the corresponding Cauchy hypersurface of the 3+1 Einstein-Klein-Gordon system a totally geodesic submanifold.
\end{remark}
Noting that all quantities above are functions of $r$ only. Thus, we can properly choose the initial data sets for $\gamma$ and $\phi$ which is compatible with \eqref{1.19} such that the second equation holds. Here, we set $\phi_t=0$. What left to be done is solving the first equation, which is just an ordinary differential equation. Noting that it is equivalent to the following equation
\begin{displaymath}
-\partial_r (e^{-2\beta}) =2r(e^{-2\beta}{\gamma_r}^2+\frac{1}{2}e^{-2\beta}{\phi_r}^2+\frac{m^2}{2} e^{-2\gamma} \phi^2)
\end{displaymath}
which is a first order linear differential equation of $e^{-2\beta}$, and with the condition $\beta(0,0)=0$, the solution takes the form
\begin{displaymath}
\beta(r) =\frac{1}{2}e^{\int_0^r (2 \xi {\gamma_r}^2+\xi {\phi_r}^2) d \xi}-\frac{1}{2}\ln{(1-\int_0^r \xi m^2 e^{-2\gamma} \phi^2 e^{\int_0^{\xi} (2 \eta {\gamma_r}^2+\eta {\phi_r}^2) d \eta} d \xi)}.
\end{displaymath}
Thus, properly choose the initial data sets for $\gamma$ and $\phi$ such that
\begin{equation} \label{cd1}
\int_0^r \xi m^2 e^{-2\gamma} \phi^2 e^{\int_0^{\xi} (2 \eta {\gamma_r}^2+\eta {\phi_r}^2) d \eta} d \xi<1,
\end{equation}
and
\begin{equation} \label{cd2}
\phi=\mathcal{O}(r^{-\frac{11}{8}}),\quad \gamma=\mathcal{O}(r^{-\frac{11}{8}})
\end{equation}
as $r \to \infty$. The first condition \eqref{cd1} guarantees that the initial energy is finite. The second condition \eqref{cd2} on asymptotic behavior guarantees the asymptotic flatness, which is to say that
$$\beta=\beta_{\infty}+\tilde{\beta}.$$
Here, $\beta_{\infty}$ is a constant and $(\tilde{\beta},K) \in  H^{s+1}_{\delta}\times H^s_{\delta+1}$ with $H^{s+1}_{\delta}$ and $H^s_{\delta+1}$ the weighted Sobolev space(see \cite{Andersson}\cite{cecile1}) where we take $\delta=-\frac{1}{2}$ and $s>1$. Thus, we finally give a non-trivial solution of the constraint equations.
\begin{remark} \label{remark1.3}
If we eliminate the condition $\gamma_t=0$, it is also easy to give a construction in a similar way by properly choosing $\gamma_t$ and $\phi_t$ which have the same asymptotic behavior as above.
\end{remark}

\subsection{The problem of global well-posedness}
The proof by Choquet-Bruhat and Geroch(see \cite{choquet1}\cite{choquet2}) of existence and uniqueness of maximal solutions to the Cauchy problem for the vacuum Einstein equations, together with the equivalence of the Cauchy data, can be generalized to our case as follows, which guaranteed the local well-posedness.
\begin{theorem} \label{thm1.5}
Let $(M_0,g_0,k,\gamma_0,\gamma_1,\phi_0,\phi_1)$ be the Cauchy data set for the 2+1 dimensional Einstein-wave-Klein-Gordon system with $U(1)$ isometry group. Then there is a unique, maximal Cauchy development $(^{(3)}M,^{(3)}g,\gamma,\phi)$ satisfying the the 2+1 dimensional Einstein-wave-Klein-Gordon system.
\end{theorem}

Our goal is to prove the following theorem about the global regularity,
\begin{theorem} \label{thm1.6}
Let $(^{(3)}M,^{(3)}g)$ be the maximal Cauchy development of a regular Cauchy data set aforementioned in section 1.3. Then there is a global in time smooth solution for the Cauchy problem of the equations.
\end{theorem}

As preparation for the above theorem, in this paper we give several results that will be useful in the proof of the global regularity. The paper is organized as follows. In Section 2, we reduce the 3+1 dimensional Einstein-Klein-Gordon system with $U(1)\times \mathbb{R}$ isometry group to a 2+1 dimensional Einstein-wave-Klein-Gordon system on $(^{(3)}M,^{(3)}g)$ with $U(1)$ symmetry. In Section 3, using vector field methods, we give energy estimates of the systems. In Section 4, we give a null coordinate system. In Section 5, we will show that the first possible singularity must occur at the axis.

\section{Reduction to the 2+1 dimensional Einstein-wave-Klein-Gordon system}
Consider the Einstein equations \eqref{1.2}, where the spacetime $(^{(4)}{M},^{(4)}{g})$ admits a spacelike translational Killing vector field. The Einstein-Klein-Gordon system can be reduced to a 2+1 dimensional Einstein-wave-Klein-Gordon system in a similar way  as we know in \cite{gudapati}. For demonstration purposes, let us go ahead and perform the reduction. Given a spacelike hypersurface orthogonal Killing vector field, the metric can be written in the following form:
\begin{equation*}
^{(4)}{g}=e^{-2\gamma}{^{(3)}g}+e^{2\gamma}{(dx^3)}^2
\end{equation*}
where $\partial_{x^3}$ is the translational spacelike Killing vector field. In a similar way as in \cite{choquet1}, the 3+1 dimensional Einstein-scalar field system \eqref{1.2} on $(^{(4)}{M},^{(4)}{g})$ can be rewritten as
\begin{equation} \label{2.1}
^{(4)}{R}_{\alpha\beta}=\tilde{R}_{\alpha\beta}-\tilde{\nabla}_{\alpha} \partial_{\beta} \gamma-\partial_{\alpha} \gamma \partial_{\beta} \gamma=\partial_{\alpha} \phi \partial_{\beta} \phi+{\tilde{g}}_{\alpha\beta} \frac{1}{2}m^2 \phi^2
\end{equation}
\begin{equation} \label{2.2}
^{(4)}{R}_{\alpha3}=0
\end{equation}
\begin{equation} \label{2.3}
^{(4)}{R}_{33}=-e^{2\gamma} \left( \tilde{g}^{\alpha\beta} \partial_\alpha \gamma \partial_\beta \gamma+\tilde{g}^{\alpha\beta} \tilde{\nabla}_\alpha \partial_\beta \gamma \right)=e^{2\gamma}\frac{m^2}{2} \phi^2.
\end{equation}
where $\tilde{g}=e^{-2\gamma}{^{(3)}g}$ and \eqref{2.2} is trivial in polarized case.

Then, we give the following formulas about the conformal transformations of Ricci tensors and the wave operator,
 \begin{equation} \label{2.4}
 \sqrt{|\det{^{(3)}g}|}=e^{3\gamma}\sqrt{|\det{\tilde{g}}|}
 \end{equation}
 \begin{equation} \label{2.5}
 \Box_{^{(3)}g} u=\frac{1}{\sqrt{|\det{^{(3)}g}|}} \partial_{\beta} \left( \sqrt{|\det{^{(3)}g}|} {^{(3)}}{g}^{\alpha\beta} \partial_{\alpha} u \right)=e^{-2\gamma} \left( \Box_{\tilde{g}} u+ {\tilde{g}}^{\alpha\beta} \partial_{\beta} \gamma \partial_{\alpha} u \right)
 \end{equation}
 \begin{equation} \label{2.6}
 ^{(3)}R_{\alpha\beta}=\tilde{R}_{\alpha\beta}-{\tilde{g}}_{\alpha\beta} {\tilde{\nabla}}^{\sigma}{\tilde{\nabla}}_{\sigma} \gamma-{\tilde{\nabla}}_{\alpha}{\tilde{\nabla}}_{\beta} \gamma+{\tilde{\nabla}}_{\alpha} \gamma{\tilde{\nabla}}_{\beta} \gamma-{\tilde{g}}_{\alpha\beta}{\tilde{\nabla}}^{\sigma} \gamma{\tilde{\nabla}}_{\sigma} \gamma
 \end{equation}
 where $\alpha,\beta,\sigma=0,1,2$.

 Rewriting the equations \eqref{2.3} under the conformally transformation and it turns out the equations take the form
 \begin{equation} \label{2.7}
 ^{(3)}{\nabla}^{\sigma} \partial_\sigma \gamma =-e^{-2\gamma}\frac{m^2}{2} \phi^2
 \end{equation}
which is a wave equation.

 Equations \eqref{2.1} rewritten in a same way results in a 2+1 dimensional Einstein equations
 \begin{equation*}
 ^{(3)}R_{\alpha\beta}=2\partial_\alpha \gamma \cdot \partial_\beta \gamma+\partial_\alpha \phi \partial_\beta \phi+2e^{-2\gamma}{^{(3)}g}_{\alpha\beta}V(\phi) \triangleq ^{(3)}\rho_{\alpha\beta}
 \end{equation*}
 This is equivalent to
 \begin{equation} \label{2.8}
 ^{(3)}G_{\alpha\beta}=^{(3)}T_{\alpha\beta}\triangleq ^{(3)}\rho_{\alpha\beta}-\frac{1}{2}tr(^{(3)}\rho) ^{(3)}g_{\alpha\beta}.
 \end{equation}
 where ${^{(3)}T}_{\alpha\beta}$ is the stress-energy tensor. And the wave equation that the scalar field satisfies can be rewritten as
 \begin{equation} \label{2.9}
 \Box_{^{(3)}g} \phi=e^{-2\gamma} m^2 \phi
 \end{equation}
 When the metric takes the form \eqref{1.3}, we give the computation of the Einstein tensor $^{(3)}G_{\alpha\beta}$ ,
\begin{eqnarray*}
&&^{(3)}G_{00}=\frac{1}{r}e^{2\alpha-2\beta} \beta_r, \\
&&^{(3)}G_{01}=\frac{1}{r} \beta_t, \\
&&^{(3)}G_{11}=\frac{1}{r} \alpha_r, \\
&&^{(3)}G_{22}=r^2 \left( e^{-2\beta} \alpha_{rr}-e^{-2\alpha} \beta_{tt}+e^{-2\beta}\alpha_r (\alpha_r-\beta_r)+e^{-2\alpha} \beta_t (\alpha_t-\beta_t) \right), \\
&&^{(3)}G_{02}=0, \\
&&^{(3)}G_{12}=0.
\end{eqnarray*}
 and the stress-energy tensor $^{(3)}T_{\alpha\beta}$,
 \begin{eqnarray*}
&&^{(3)}T_{00}={\gamma_t}^2+e^{2\alpha-2\beta}{\gamma_r}^2+\frac{1}{2}{\phi_t}^2+\frac{1}{2} e^{2\alpha-2\beta} {\phi_r}^2+\frac{m^2}{2}e^{2\alpha-2\gamma}\phi^2 \\
&&=e^{2\alpha} \mathbf{e},\\
&&^{(3)}T_{01}=2\gamma_t \gamma_r+\phi_t\phi_r=e^{(\alpha+\beta)} \mathbf{m}, \\
&&^{(3)}T_{11}=e^{2\beta-2\alpha}{\gamma_t}^2+{\gamma_r}^2+\frac{1}{2}e^{2\beta-2\alpha}{\phi_t}^2+\frac{1}{2} {\phi_r}^2-\frac{m^2}{2}e^{2\beta-2\gamma}\phi^2,\\
&&^{(3)}T_{22}=r^2 e^{-2\alpha}{\gamma_t}^2-r^2 e^{-2\beta}{\gamma_r}^2+\frac{r^2}{2}e^{-2\alpha}{\phi_t}^2-\frac{r^2}{2}e^{-2\beta} {\phi_r}^2-\frac{m^2}{2}r^2 e^{-2\gamma}\phi^2,\\
&&^{(3)}T_{02}=0, \\
&&^{(3)}T_{12}=0.
\end{eqnarray*}

 Thus, using the translational Killing vector field, we have reduced the 3+1 dimensional Einstein-Klein-Gordon system \eqref{1.1} to a 2+1 dimensional Einstein-wave-Klein-Gordon system \eqref{2.7}\eqref{2.8}\eqref{2.9}.

 Now we write the 2+1 dimensional radially symmetric Einstein-wave-Klein-Gordon system in local coordinates,
 \begin{equation} \label{2.10}
{\beta}_r=re^{2\beta-2\alpha}{\gamma_t}^2+r{\gamma_r}^2+\frac{r}{2}e^{2\beta-2\alpha}{\phi_t}^2+\frac{r}{2}{\phi_r}^2+\frac{m^2}{2}re^{2\beta-2\gamma} \phi^2
\end{equation}
\begin{equation} \label{2.11}
\beta_t=2r\gamma_t \gamma_r+r \phi_t \phi_r
\end{equation}
\begin{equation} \label{2.12}
\alpha_r=re^{2\beta-2\alpha}{\gamma_t}^2+r{\gamma_r}^2+\frac{r}{2}e^{2\beta-2\alpha}{\phi_t}^2+\frac{r}{2}{\phi_r}^2-\frac{m^2}{2}re^{2\beta-2\gamma} \phi^2
\end{equation}
\begin{eqnarray} \label{2.13}
 &&e^{-2\beta} \alpha_{rr}-e^{-2\alpha} \beta_{tt}+e^{-2\beta} \alpha_r(\alpha_r-\beta_r)+e^{-2\alpha} \beta_t (\alpha_t-\beta_t) \\ \nonumber
 &&=-\frac{m^2}{2}e^{-2\gamma} \phi^2+ e^{-2\alpha}({\gamma_t}^2+\frac{1}{2}{\phi_t}^2)-e^{-2\beta}({\gamma_r}^2+\frac{1}{2}{\phi_r}^2)
 \end{eqnarray}
\begin{eqnarray} \label{2.14}
\Box_{^{(3)}g}\gamma &=& -e^{-2\alpha}(\gamma_{tt}+(\beta_t-\alpha_t) \gamma_t)+e^{-2\beta} (\gamma_{rr}+\frac{\gamma_r}{r}+(\alpha_r-\beta_r) \gamma_r) \\ \nonumber
&=& -\frac{m^2}{2}e^{-2\gamma}\phi^2
\end{eqnarray}
\begin{eqnarray} \label{2.15}
\Box_{^{(3)}g} \phi &=& -e^{-2\alpha}(\phi_{tt}+(\beta_t-\alpha_t) \phi_t)+e^{-2\beta} (\phi_{rr}+\frac{\phi_r}{r}+(\alpha_r-\beta_r) \phi_r) \\ \nonumber
&=& m^2e^{-2\gamma} \phi.
\end{eqnarray}
Therefore, as we mentioned in section 1, we can set $\beta(0,0)=0$. Then, by \eqref{2.11}, we have $\beta(t,0)=0$ on $\Gamma$.

In addition, we mention here that, if we consider the case that the scalar field is massless, i.e. $V=0$, then we can remove the condition that the Killing vector field $\partial_{x^3}$ is hypersurface orthogonal, and the metric will take the general form
\begin{equation*}
^{(4)}{g}=e^{-2\gamma}{^{(3)}}g+e^{2\gamma}{(dx^3+A_{\alpha}dx^{\alpha})}^2.
\end{equation*}
Thus, the reduction in \cite{gudapati} yields a wave map equations coupled with a wave equation.

Noting that if $V=0$, then $\alpha_r=\beta_r$ by the 2+1 dimensional Einstein equations, then $e^{\alpha-\beta}$ is independent of $r$. Let
 \begin{displaymath}
 T(t)=\int_0^t e^{\alpha-\beta} (\tau) d \tau.
 \end{displaymath}
Therefore, we get a coordinate $(T,r)$ in which $(\mathcal{Q},\check{g})$ is conformally flat. For simplicity, we still denote $T$ by $t$. Thus, rewritting the Einstein-wave map-scalar field system in the new coordinates, we can obtain a spherically symmetric wave map to the hyperbolic space $(\mathbb{H}^2,h)$ coupled with a linear wave equation on the Minkowski spacetimes $(\mathbb{R}^{1+2},m)$.

Noting that regularity for the linear wave equation $\Box_m \phi=0$ that the scalar fields satisfies is trivial , we only need to consider the spherically symmetric wave map on the Minkowski spacetimes $(\mathbb{R}^{1+2},m)$, which can be proved applying \cite{christodoulou2} in the system.

\section{Energy estimates}
 As a preliminary part of later work, we will give the energy estimates of the Einstein-wave-Klein-Gordon system in this section. The energy estimates are performed in a well-known way as in \cite{alinhac}. The notations here will follow what were given in \cite{Andersson}. For simplicity, we denote $(^{(3)}M,^{(3)}g)$ by $(M,g)$ from now on.

Let us define the energy on the Cauchy surface $\Sigma_t$,
\begin{eqnarray*}
E(t)&:=&\int_{\Sigma_t} \mathbf{e} {\bar{\mu}}_q \\
&=& 2\pi \int_0^{\infty} \mathbf{e}(t,r) r e^{\beta(t,r)} dr,
\end{eqnarray*}
the energy in a coordinate ball $B_r$,
\begin{eqnarray*}
E(t,r)& :=& \int_{B_r} \mathbf{e} {\bar{\mu}}_q \\
&=&2\pi \int_0^r \mathbf{e}(t,r') r' e^{\beta(t,r')} d r',
\end{eqnarray*}
the energy inside the causal past $J^{-}(O)$ of $O$,
\begin{displaymath}
E^O(t):=\int_{{\Sigma_t} \cap {J^-(O)}} \mathbf{e} {\bar{\mu}}_q
\end{displaymath}
with $O$ appears at the axis.
\subsection{Energy conservation}
We start by proving the energy is conserved.
\begin{proposition} \label{prop3.1}
The energy $E(t)$ is conserved,
\begin{displaymath}
\frac{d}{dt} E(t)=0.
\end{displaymath}
\end{proposition}
\begin{proof}
Consider two Cauchy surfaces $\Sigma_s$ and $\Sigma_{\tau}$ at $t=s$ and $t=\tau$ respectively. First, let us construct a divergence free vector field $P_T$. Consider the Einstein's equations \eqref{2.10} and \eqref{2.11}. They can be rewritten as follows
\begin{equation} \label{3.1}
-\partial_r (e^{-\beta})=r e^{\beta} \mathbf{e}
\end{equation}
\begin{equation} \label{3.2}
-\partial_t (e^{-\beta})=r e^{\alpha} \mathbf{m}
\end{equation}
From the smoothness of $\beta$, we have
\begin{displaymath}
-\partial^2_{rt} (e^{-\beta})=-\partial^2_{tr} (e^{-\beta}).
\end{displaymath}
Together with \eqref{3.1}\eqref{3.2}, we infer
\begin{equation} \label{3.3}
-\partial_t (r e^{\beta} \mathbf{e})+\partial_r(r e^{\alpha} \mathbf{m})=0.
\end{equation}
Now we define a vector field
\begin{displaymath}
P_T:=-e^{-\alpha} \mathbf{e} \partial_t+e^{-\beta} \mathbf{m} \partial_r.
\end{displaymath}
The divergence of $P_T$ is given by
\begin{eqnarray} \label{3.4}
\nabla_{\nu} P_T^{\nu}&=&\frac{1}{\sqrt{|g|}} \partial_{\nu} \left( \sqrt{|g|} P_T^{\nu} \right) \\ \nonumber
&=& \frac{1}{r e^{\beta+\alpha}} \left(-\partial_t (r e^{\beta} \mathbf{e})+\partial_r(re^{\alpha} \mathbf{m})\right).
\end{eqnarray}
By \eqref{3.3}, we know that $P_T$ is divergence free. Next, let us apply Stokes' theorem in the space-time region whose boundary is $\Sigma_s \cup \Sigma_{\tau}$. Due to the asymptotic flatness, the boundary terms at $r=\infty$ do not contribute. Thus, we have
\begin{equation} \label{3.5}
\int_{\Sigma_s} e^{\alpha} P_T^t {\bar{\mu}}_q-\int_{\Sigma_{\tau}} e^{\alpha} P_T^t {\bar{\mu}}_q=0.
\end{equation}
Therefore, it follows that
\begin{equation} \label{3.6}
E(\tau)=E(s)
\end{equation}
which proves that the energy is conserved.
\end{proof}
With the conservation of the energy, we can prove that the metric functions $\beta(t,r)$ and $\alpha(t,r)$ are uniformly bounded during the evolution of the Einstein-wave-Klein-Gordon system.
\begin{proposition} \label{prop3.2}
There exists constants $c_{\beta}^{-},c_{\beta}^{+},c_{\alpha}^{-},c_{\alpha}^{+}$ depending only on the initial data and the universal constants such that the following uniform bounds on the metric functions $\beta(t,r)$ and $\alpha(t,r)$ hold
\begin{displaymath}
c_{\beta}^{-} \leq \beta(t,r) \leq c_{\beta}^{+},
\end{displaymath}
\begin{displaymath}
c_{\alpha}^{-} \leq \alpha(t,r) \leq c_{\alpha}^{+}.
\end{displaymath}
\end{proposition}
\begin{proof}
For simplicity of notation, we use a generic constant $c$ for the estimates on $\beta(t,r)$ and $\alpha(t,r)$. Integrating \eqref{3.1} with respect to $r$ and noting the normalisation that $\beta|_{\Gamma}=0$ as we mentioned before, we obtain
\begin{displaymath}
1-e^{-\beta}=\int_0^r \mathbf{e} r' e^{\beta} d r'=\frac{1}{2\pi} E(t,r),
\end{displaymath}
which implies
\begin{displaymath}
e^{\beta}={\left( 1- \frac{1}{2\pi} E(t,r) \right)}^{-1}.
\end{displaymath}
Now we define
\begin{displaymath}
\beta_{\infty}(t)=\lim_{r \to \infty} \beta(r,t).
\end{displaymath}
Then we have
\begin{displaymath}
e^{\beta_{\infty}(t)}={\left( 1- \frac{1}{2\pi} E(t) \right)}^{-1}.
\end{displaymath}
Since $E(t,r)$ is a nondecreasing function of $r$, then so is $\beta(t,r)$, therefore,
\begin{displaymath}
1=e^{\beta(t,0)} \leq e^{\beta(t,r)} \leq e^{\beta_{\infty}(t)}.
\end{displaymath}
Moreover, since the energy is conserved $E(t)=E(0)$, $\beta_{\infty}(t)$ is also conserved during the evolution of the Einstein-wave-Klein-Gordon system,
$$\beta_{\infty}(t)=\beta_{\infty}(0).$$
Thus,
\begin{displaymath}
0 \leq \beta(t,r) \leq \beta_{\infty}(0).
\end{displaymath}
Now we introduce
\begin{displaymath}
\mathbf{f}:=m^2 e^{-2\gamma} \phi^2
\end{displaymath}
Thus the Einstein's equation \eqref{2.12} for $\alpha_r$ takes the form
\begin{displaymath}
\alpha_r=r e^{2\beta} (\mathbf{e}-\mathbf{f}).
\end{displaymath}
Similarly, we integrate the above equation with respect to $r$ and noting the normalisation that $\alpha|_{\Gamma}=0$, we obtain
\begin{eqnarray*}
\alpha(t,r) & \leq & c \int_0^r (\mathbf{e}-\mathbf{f}) r' e^{\beta} d r' \\
& \leq & c \int_0^r \mathbf{e} r' e^{\beta} d r' \\
& \leq & c
\end{eqnarray*}
and
\begin{eqnarray*}
\alpha(t,r) &\geq& -c\int_0^r \frac{\mathbf{f}}{2} r' e^{\beta} d r' \\
& \geq & -c \int_0^r \mathbf{e} r' e^{\beta} d r' \\
& \geq& -c.
\end{eqnarray*}
This finishes the proof of the proposition.
\end{proof}
\subsection{The Vector field Method}
Let $X$ be a vector field on $M$. Set the corresponding momentum $P_X$ as follows
\begin{equation} \label{3.7}
P_X^{\mu}=T^{\mu}_{\nu} X^{\nu},
\end{equation}
then, we have
\begin{equation} \label{3.8}
\nabla_{\nu} P_X^{\nu}=X^{\mu} \nabla_{\nu} T^{\nu}_{\mu}+T^{\nu}_{\mu} \nabla_{\nu} X^{\mu}.
\end{equation}
Since the stress-energy tensor $T_{\mu\nu}$ satisfies
\begin{displaymath}
\nabla^{\mu} T_{\mu\nu}=0,
\end{displaymath}
the first term in the right hand side vanishes, hence
\begin{eqnarray*}
\nabla_{\nu} P_X^{\nu}&=&T^{\mu\nu} \nabla_{\mu} X_{\nu} \\
&=& \frac{1}{2}{ ^{(X)}{\pi}_{\mu\nu} T^{\mu\nu}},
\end{eqnarray*}
where the deformation tensor $ ^{(X)} {\pi}_{\mu\nu}$ is defined by
\begin{eqnarray*}
 ^{(X)} {\pi}_{\mu\nu} &:=& \nabla_{\mu} X_{\nu}+\nabla_{\nu} X_{\mu} \\
 &=& g_{\sigma\nu} \partial_{\mu} X^{\sigma}+g_{\sigma\mu} \partial_{\nu} X^{\sigma}+X^{\sigma} \partial_{\sigma} g_{\mu\nu}.
\end{eqnarray*}

For instance, consider $T=e^{-\alpha}\partial_t$, the corresponding momentum $P_T$ is
\begin{equation} \label{3.9}
P_T=-e^{-\alpha} \mathbf{e} \partial_t+e^{-\beta} \mathbf{m} \partial_r.
\end{equation}
Then,  we have the non-zero components of the deformation tensor are
\begin{displaymath}
 ^{(T)} {\pi}_{01} = e^{\alpha} \alpha_r= ^{(T)} {\pi}_{10},
\end{displaymath}
\begin{displaymath}
 ^{(T)} {\pi}_{11} =2 e^{2\beta-\alpha} \beta_t.
\end{displaymath}
Thus, using the Einstein equations \eqref{2.11}\eqref{2.12}, we have that the divergence of $P_T$ is,
\begin{eqnarray*}
\nabla_{\nu} P_T^{\nu}&=&e^{-\alpha} \beta_t (\mathbf{e}-\mathbf{f})-e^{-\beta}\alpha_r \mathbf{m} \\
&=& 0.
\end{eqnarray*}
This is compatible with \eqref{3.3}\eqref{3.4}.

Now let $J^-(O)$ be the causal past of the point $O$ on the axis and $I^-(O)$ the chronological past of $O$. Compared to the flat case, we give the following definitions
\begin{displaymath}
\Sigma_t^O:=\Sigma_t \cap J^-(O),
\end{displaymath}
\begin{displaymath}
K(t):=\cup_{0 \leq t \leq t' < t_O} \Sigma_{t'} \cap J^-(O),
\end{displaymath}
\begin{displaymath}
C(t):=\cup_{0 \leq t \leq t' < t_O} \Sigma_{t'} \cap \left(J^-(O) \setminus I^-(O) \right),
\end{displaymath}
\begin{displaymath}
K(t,s):=\cup_{0 \leq t \leq t' < s} \Sigma_{t'} \cap J^-(O),
\end{displaymath}
\begin{displaymath}
C(t,s):=\cup_{0 \leq t \leq t' < s} \Sigma_{t'} \cap \left(J^-(O) \setminus I^-(O) \right)
\end{displaymath}
for $0 \leq t<s<t_O$.

Then, the volume 3-form of $(M,g)$ is given by
\begin{displaymath}
{\bar{\mu}}_g=r e^{\beta+\alpha} dt \wedge dr \wedge d \theta
\end{displaymath}
and the area 2-form of $(\Sigma,q)$ is given by
\begin{displaymath}
{\bar{\mu}}_q=r e^{\beta} dr \wedge d \theta.
\end{displaymath}
Let us define 1-forms $\tilde{l},\tilde{n}$ and $\tilde{m}$ as follows
\begin{displaymath}
\tilde{l}:=-e^{\alpha}dt+e^{\beta}dr,
\end{displaymath}
\begin{displaymath}
\tilde{n}:=-e^{\alpha}dt-e^{\beta}dr,
\end{displaymath}
\begin{displaymath}
\tilde{m}:=r d \theta,
\end{displaymath}
therefore,
\begin{displaymath}
{\bar{\mu}}_g=\frac{1}{2} \left( \tilde{l} \wedge \tilde{n} \wedge \tilde{m} \right).
\end{displaymath}
Then we introduce the 2-forms ${\bar{\mu}}_{\tilde{l}}$ and ${\bar{\mu}}_{\tilde{n}}$ such that
\begin{displaymath}
{\bar{\mu}}_{\tilde{l}}:=-\frac{1}{2} \tilde{n} \wedge \tilde{m},
\end{displaymath}
\begin{displaymath}
{\bar{\mu}}_{\tilde{n}}:=\frac{1}{2} \tilde{l} \wedge \tilde{m},
\end{displaymath}
so we have
\begin{displaymath}
{\bar{\mu}}_g=-\tilde{l} \wedge {\bar{\mu}}_{\tilde{l}},
\end{displaymath}
\begin{displaymath}
{\bar{\mu}}_g=-\tilde{n} \wedge {\bar{\mu}}_{\tilde{n}}.
\end{displaymath}
Now, let us apply the Stokes' theorem for the ${\bar{\mu}}_g$-divergence of $P_X$ in the region $K(\tau,s)$. We have
\begin{equation} \label{3.10}
\int_{K(\tau,s)} \nabla_{\nu} P_X^{\nu} {\bar{\mu}}_g= \int_{\Sigma_s^O} e^{\alpha} P_X^t {\bar{\mu}}_q-\int_{\Sigma_{\tau}^O} e^{\alpha}P_X^t {\bar{\mu}}_q+Flux(P_X)(\tau,s)
\end{equation}
where
\begin{displaymath}
Flux(P_X)(\tau,s)=-\int_{C(\tau,s)} \tilde{n}(P_X) {\bar{\mu}}_{\tilde{n}}.
\end{displaymath}

\subsection{Monotonicity of Energy}
\begin{proposition} \label{prop3.3}
There holds $E^O(\tau) \geq E^O(s)$ for $0 \leq \tau<s<t_O$.
\end{proposition}
\begin{proof}
Applying Stokes' theorem \eqref{3.10} to the vector field $P_T$, we have
\begin{equation} \label{3.11}
0=-\int_{\Sigma_s^O} \mathbf{e} {\bar{\mu}}_q+\int_{\Sigma_{\tau}^O} \mathbf{e} {\bar{\mu}}_q+Flux(P_T)(\tau,s)
\end{equation}
where
\begin{eqnarray*}
Flux(P_T)(\tau,s) &=& -\int_{C(\tau,s)} \tilde{n}(P_T) {\bar{\mu}}_{\tilde{n}} \\
&=&-\int_{C(\tau,s)} (\mathbf{e}-\mathbf{m}){\bar{\mu}}_{\tilde{n}}.
\end{eqnarray*}
Noting that we have $\mathbf{e} \geq |\mathbf{m}|$, we obtain
\begin{displaymath}
Flux(P_T)(\tau,s) \leq 0
\end{displaymath}
which implies
\begin{displaymath}
E^O(\tau)-E^O(s) \geq 0,\quad \forall 0\leq \tau \leq s <t_O.
\end{displaymath}
This concludes the proof of the proposition.
\end{proof}

\section{Null coordinates}
In this section, we  introduce a null coordinate system, in which the wave equations may be written in a classical form in the flat case. This coordinate system will be of great importance in our work on global well-posedness. In the following part, we assume that all objects are smooth, unless otherwise stated.
\subsection{Existence of null coordinates}
 Let $(\mathcal{Q},\check{g})$ be the orbit space, where
 \begin{displaymath}
 \mathcal{Q}=M/{{\mathbb{S}}^1}
 \end{displaymath}
 and
 \begin{equation*}
 \check{g}=-e^{2\alpha} dt^2+e^{2\beta} dr^2.
 \end{equation*}
 As discussed in Section 1, the orbit space $(\mathcal{Q},\check{g})$ is a 2-dimensional globally hyperbolic Lorentzian space and thus in particular locally conformally flat. Hence, as noted in \cite{Rendall}, we may introduce a null coordinate system with respect to which $\check{g}$ takes the form
 \begin{displaymath}
 \check{g}=-e^{2\lambda}(u,v)dudv
 \end{displaymath}
 which means that the 3-dimensional manifold $(M,g)$ admits a coordinate system $(u,v,\theta)$ such that $g$ takes the form
  \begin{displaymath}
g=-e^{2\lambda}(u,v)dudv+r^2(u,v)d\theta^2
 \end{displaymath}
 where now $d\theta^2$ is the line element on the $\mathbb{S}^1$ symmetry orbit. Then we can define
 \begin{displaymath}
 T=\frac{u+v}{2},\, R=\frac{v-u}{2}.
 \end{displaymath}

 A direct calculation gives the following equations that the null coordinates satisfy,
 \begin{equation} \label{4.1.1}
 \begin{cases}
 u_t+e^{\alpha-\beta} u_r=0, \\
 v_t-e^{\alpha-\beta} v_r=0.
 \end{cases}
 \end{equation}
 Obviously, such coordinate system exists locally. Now we impose the following initial boundary conditions,
 \begin{equation} \label{4.1.2}
 \begin{cases}
t=0: u=-r,\, v=r\\
r=0: u=v
 \end{cases}
 \end{equation}

 We will show that the solution of the above initial-boundary problem \eqref{4.1.1}\eqref{4.1.2} give a global null coordinate system in the next part, where a prior estimate for the Jacobian is given.

\subsection{$L^{\infty}$ Estimate for the Jacobian}
Based on the energy estimates in section 3, we aim to derive uniform bounds for the Jacobian transformation between $(t,r,\theta)$ and $(u,v,\theta)$ coordinates in this section, which can lead to a global null coordinate system. In view of the definitions of the 1-forms $\tilde{l}$ and $\tilde{n}$ in section 3, their corresponding vector fields are null, given by
\begin{displaymath}
\tilde{l}=e^{-\alpha} \partial_t+e^{-\beta} \partial_r,
\end{displaymath}
\begin{displaymath}
\tilde{n}=e^{-\alpha} \partial_t-e^{-\beta} \partial_r.
\end{displaymath}
\begin{lemma} \label{lem3.4}
There exists two scalar functions $\mathcal{F}$ and $\mathcal{G}$ such that
\begin{equation} \label{3.12}
\partial_v=\frac{1}{2} e^{\mathcal{F}} \tilde{l},\, \partial_u=\frac{1}{2} e^{\mathcal{G}} \tilde{n}.
\end{equation}
Moreover, $\mathcal{F}$ and $\mathcal{G}$ satisfy the following equations
\begin{equation} \label{3.13}
2\partial_v (\mathcal{G})=e^{\mathcal{F}} r e^{\beta}(\mathbf{e}+\mathbf{m}-\mathbf{f}),
\end{equation}
\begin{equation} \label{3.14}
2\partial_u (\mathcal{F})=-e^{\mathcal{G}} r e^{\beta}(\mathbf{e}-\mathbf{m}-\mathbf{f}).
\end{equation}
\end{lemma}
\begin{proof}
In view of \eqref{4.1.1} and \eqref{4.1.2}, on the initial surface, there holds
\begin{eqnarray*}
\left(
\begin{matrix}
\partial_t u & \partial_r u  \\
\partial_t v & \partial_r v
\end{matrix}\right)
&=&\left(
\begin{matrix}
e^{\alpha-\beta} &-1 \\
e^{\alpha-\beta} & 1
\end{matrix}\right)
\end{eqnarray*}
which implies
\begin{eqnarray*}
\left(
\begin{matrix}
\partial_u t & \partial_v t  \\
\partial_u r & \partial_v r
\end{matrix}\right)
&=& \frac{1}{2}\left(
\begin{matrix}
e^{\beta-\alpha} &e^{\beta-\alpha} \\
-1 & 1
\end{matrix}\right)
\end{eqnarray*}
on $t=0$.

Thus, we have
\begin{displaymath}
\partial_v r=\frac{1}{2},\, \partial_u r=-\frac{1}{2}, \, \tilde{l}(r)=e^{-\beta},\, \tilde{n}(r)=-e^{-\beta}
\end{displaymath}
on $t=0$.

Note the fact that $\partial_u$ and $\partial_v$ are null vectors, and $\partial_u,\partial_v,\tilde{l}$ and $\tilde{n}$ are all future directed. Thus, we infer that there exists two scalar functions $\mathcal{F}$ and $\mathcal{G}$ such that
\begin{displaymath}
\partial_v=\frac{1}{2} e^{\mathcal{F}} \tilde{l},\, \partial_u=\frac{1}{2} e^{\mathcal{G}} \tilde{n},
\end{displaymath}
with the normalization on the initial Cauchy surface
\begin{displaymath}
\mathcal{F}=\mathcal{G}=\beta.
\end{displaymath}

Now we derive the equations that $\mathcal{F}$ and $\mathcal{G}$ satisfy. We have
\begin{displaymath}
[\tilde{l},\tilde{n}]=2 e^{-(\beta+\alpha)} (-\alpha_r \partial_t+\beta_t \partial_r).
\end{displaymath}
Then,
\begin{eqnarray*}
[\partial_v,\partial_u]&=&\frac{e^{(\mathcal{F}+\mathcal{G})}}{4} \left( [\tilde{l},\tilde{n}]+\tilde{l}(\mathcal{G}) \tilde{n}-\tilde{n}(\mathcal{F}) \tilde{l}\right) \\
&=& \frac{e^{(\mathcal{F}+\mathcal{G})}}{2} e^{-(\beta+\alpha)} (-\alpha_r \partial_t+\beta_t \partial_r)+\frac{e^{\mathcal{G}}}{2} \partial_v(\mathcal{G}) (e^{-\alpha}\partial_t-e^{-\beta}\partial_r) \\
&&-\frac{e^{\mathcal{F}}}{2} \partial_u(\mathcal{F}) (e^{-\alpha}\partial_t+e^{-\beta} \partial_r).
\end{eqnarray*}
Since $[\partial_v,\partial_u]=0$, $\mathcal{F}$ and $\mathcal{G}$ are such that
\begin{displaymath}
e^{-\mathcal{F}} \partial_v(\mathcal{G})-e^{-\mathcal{G}} \partial_u(\mathcal{F})=r e^{\beta} (\mathbf{e}-\mathbf{f}),
\end{displaymath}
\begin{displaymath}
e^{-\mathcal{F}} \partial_v(\mathcal{G})+e^{-\mathcal{G}} \partial_u(\mathcal{F})=r e^{\beta} \mathbf{m},
\end{displaymath}
and hence
\begin{displaymath}
2\partial_v(\mathcal{G})=e^{\mathcal{F}} r e^{\beta} (\mathbf{e}+\mathbf{m}-\mathbf{f}),
\end{displaymath}
\begin{displaymath}
2\partial_u(\mathcal{F})=-e^{\mathcal{G}} r e^{\beta} (\mathbf{e}-\mathbf{m}-\mathbf{f}).
\end{displaymath}
This concludes the proof of the lemma.
\end{proof}

Now, with respect to the null coordinate system, let us revisit Stokes' theorem for ${\bar{\mu}}_g$-divergence for $P_X$ in $K(\tau,s)$. We have
\begin{displaymath}
d v=-e^{-\mathcal{F}} \tilde{n},\, du=-e^{-\mathcal{G}} \tilde{l}.
\end{displaymath}
And the volume 3-form of $(M,g)$ takes the form
\begin{displaymath}
{\bar{\mu}}_g=\frac{1}{2} r e^{2\lambda} du \wedge dv \wedge d \theta.
\end{displaymath}
Next, we introduce the 2-forms ${\bar{\mu}}_v$ and ${\bar{\mu}}_u$ as follows
\begin{displaymath}
{\bar{\mu}}_g=dv \wedge {\bar{\mu}}_v,\, {\bar{\mu}}_g=du \wedge {\bar{\mu}}_u.
\end{displaymath}
From the above two formulas, we infer
\begin{displaymath}
{\bar{\mu}}_v=-\frac{1}{2} r e^{2\lambda} ( du \wedge d \theta),\, {\bar{\mu}}_u=\frac{1}{2}r e^{2\lambda} (dv \wedge d \theta).
\end{displaymath}
Therefore,
\begin{displaymath}
Flux(P_X)(\tau,s)=\int_{C(\tau,s)} dv(P_X) {\bar{\mu}}_v,
\end{displaymath}
for instance,
\begin{eqnarray*}
Flux(P_T)(\tau,s)&=&\int_{C(\tau,s)} dv(P_T) {\bar{\mu}}_v, \\
&=&-\int_{C(\tau,s)} e^{-\mathcal{F}} (\mathbf{e}-\mathbf{m}) {\bar{\mu}}_v.
\end{eqnarray*}
\begin{lemma} \label{lem3.5}
There exists constants $c_{\mathcal{G}}^{-},c_{\mathcal{G}}^{+},c_{\mathcal{F}}^{-}$ and $c_{\mathcal{F}}^{+}$ depending only on the initial data and the universal constants such that the following uniform bounds hold
\begin{displaymath}
c_{\mathcal{G}}^{-} \leq \mathcal{G} \leq c_{\mathcal{G}}^{+}
\end{displaymath}
\begin{displaymath}
c_{\mathcal{F}}^{-} \leq \mathcal{F} \leq c_{\mathcal{F}}^{+}.
\end{displaymath}
\end{lemma}
\begin{proof}
Integrating \eqref{3.14} with the fact that $\mathcal{F}=\mathcal{G}=\beta$ on the initial Cauchy surface, we obtain
\begin{displaymath}
2\mathcal{F}(u,v)-2\beta(-v,v)=\int_{-v}^u e^{\mathcal{G}} r e^{\beta} (\mathbf{e}-\mathbf{m}-\mathbf{f})(u',v) d u'.
\end{displaymath}
Next, note that
\begin{displaymath}
-\frac{e^{2\lambda}}{2}=g(\partial_u,\partial_v)=\frac{e^{\mathcal{F}+\mathcal{G}}}{4} g(\tilde{n},\tilde{l})=-\frac{e^{\mathcal{F}+\mathcal{G}}}{2}
\end{displaymath}
and we infer
\begin{equation} \label{3.15}
2\lambda=\mathcal{F}+\mathcal{G}.
\end{equation}
Thus, we have
\begin{displaymath}
2\mathcal{F}(u,v)-2\beta(-v,v)=\int_{-v}^u e^{-\mathcal{F}} r e^{\beta} (\mathbf{e}-\mathbf{m}-\mathbf{f})e^{2\lambda} d u'.
\end{displaymath}
Noting the fact that $|\mathbf{e}\pm \mathbf{m}-\mathbf{f}| \leq \mathbf{e} \pm \mathbf{m}$ and $-v \leq u$,  using Proposition 3.2, we obtain
\begin{displaymath}
|\mathcal{F}(u,v)| \lesssim c+\int_{-v}^u e^{-\mathcal{F}} (\mathbf{e}-\mathbf{m}) r e^{2\lambda} du'.
\end{displaymath}
Since $du=-e^{-\mathcal{G}}\tilde{l}$, we obtain
\begin{displaymath}
|\mathcal{F}(u,v)| \lesssim c+\int_{-v}^u dv(P_T) r e^{2\lambda} du'.
\end{displaymath}
After integration in $\theta$, the right-hand sides are bounded by fluxes which in turn are bounded by the energy, and hence
\begin{displaymath}
|\mathcal{F}| \leq c.
\end{displaymath}

For $\mathcal{G}$, if $u \leq 0$, we can integrate \eqref{3.13} from the initial surface and get the uniform bound in a similar way as we did for $\mathcal{F}$.

If $u>0$, we need some normalisation on $\Gamma$ to estimate $\mathcal{G}$. Noting that we have $u=v$, which infers $R=r=0$ and $u=v=T$. In view of \eqref{4.1.1} and \eqref{4.1.2}, on $\Gamma$, there holds
\begin{eqnarray*}
\left(
\begin{matrix}
\partial_t u & \partial_r u  \\
\partial_t v & \partial_r v
\end{matrix}\right)
&=&\left(
\begin{matrix}
T_t &T_t \\
T_t & -T_t
\end{matrix}\right)
\end{eqnarray*}
which implies
\begin{eqnarray*}
\left(
\begin{matrix}
\partial_u t & \partial_v t  \\
\partial_u r & \partial_v r
\end{matrix}\right)
&=& \frac{1}{2}\left(
\begin{matrix}
{T_t}^{-1} &{T_t}^{-1} \\
{T_t}^{-1} & -{T_t}^{-1}
\end{matrix}\right)
\end{eqnarray*}
on $\Gamma$.

Thus, we have
\begin{displaymath}
\partial_v r=-\partial_u r, \, \tilde{l}(r)=1,\, \tilde{n}(r)=-1
\end{displaymath}
on $\Gamma$. And by Lemma 4.1, we have
\begin{displaymath}
\partial_u r=-e^{\mathcal{G}-\beta},\quad \partial_v r=e^{\mathcal{F}-\beta}.
\end{displaymath}
Therefore, we obtain that on $\Gamma$ there holds
$$\mathcal{G}=\mathcal{F}.$$

Now we can integrate \eqref{3.13} from the axis,
\begin{eqnarray*}
2\mathcal{G}(u,v)-2\mathcal{G}(u,u)&=& 2\mathcal{G}(u,v)-2\mathcal{F}(u,u) \\
&=&\int_u^v e^{\mathcal{F}} r e^{\beta} (\mathbf{e}+\mathbf{m}-\mathbf{f})(u,v') d v'\\
&=&\int_u^v e^{-\mathcal{G}} r e^{\beta} (\mathbf{e}+\mathbf{m}-\mathbf{f})e^{2\lambda} d v'
\end{eqnarray*}
Noting the fact that $|\mathbf{e}\pm \mathbf{m}-\mathbf{f}| \leq \mathbf{e} \pm \mathbf{m}$ and $u \leq v$, using Proposition 3.2, we deduce
\begin{displaymath}
|\mathcal{G}(u,v)| \lesssim c+\int_u^v e^{-\mathcal{G}} (\mathbf{e}+\mathbf{m}) r e^{2\lambda} dv',
\end{displaymath}
Since $dv=-e^{-\mathcal{F}}\tilde{n}$, we obtain
\begin{eqnarray*}
|\mathcal{G}(u,v)| &\lesssim& c+\int_u^v du(P_T) r e^{2\lambda} dv' \\
& \leq & c.
\end{eqnarray*}

This finishes the proof of the lemma.
\end{proof}

\begin{corr} \label{cor3.7}
There exist constants $c_{\lambda}^{-}$ and $c_{\lambda}^{+}$ depending only on the initial energy and the universal constants such that the following uniform bounds hold on the metric function $\lambda$ in null coordinates
\begin{equation} \label{3.17}
c_{\lambda}^{-} \leq \lambda \leq c_{\lambda}^{+}.
\end{equation}
\end{corr}
\begin{proof}
This follows immediately from \eqref{3.15} and Lemma 4.2.
\end{proof}

Now, by \eqref{3.12}, we can write the Jacobian $\mathbf{J}$ of the transition functions between $(t,r,\theta)$ and $(v,u,\theta)$,
\begin{eqnarray*}
\mathbf{J} &:=&\left(
\begin{matrix}
\partial_v t & \partial_u t & \partial_{\theta} t \\
\partial_v r & \partial_u r & \partial_{\theta} r \\
\partial_v \theta & \partial_u \theta & \partial_{\theta} \theta
\end{matrix}\right) \\
&=& \frac{1}{2}\left(
\begin{matrix}
e^{\mathcal{F}-\alpha} &e^{\mathcal{G}-\alpha} &0 \\
e^{\mathcal{F}-\beta} &-e^{\mathcal{G}-\beta} & 0 \\
0 & 0& 2
\end{matrix}\right)
\end{eqnarray*}
then the inverse Jacobian ${\mathbf{J}}^{-1}$ is given by
\begin{eqnarray*}
{\mathbf{J}}^{-1}&=&\left(
\begin{matrix}
\partial_t v& \partial_r v & \partial_{\theta} v \\
\partial_t u & \partial_r u & \partial_{\theta} u \\
\partial_t \theta & \partial_r \theta & \partial_{\theta} \theta
\end{matrix}\right) \\
&=& \left(
\begin{matrix}
e^{-\mathcal{F}+\alpha} &e^{-\mathcal{F}+\beta} &0 \\
e^{-\mathcal{G}+\alpha} &-e^{-\mathcal{G}+\beta} & 0 \\
0 & 0& 1
\end{matrix}\right)
\end{eqnarray*}
Therefore,
\begin{displaymath}
dv=e^{(-\mathcal{F}+\alpha)}dt+e^{(-\mathcal{F}+\beta)}dr,\quad du=e^{(-\mathcal{G}+\alpha)}dt-e^{(-\mathcal{G}+\beta)}dr.
\end{displaymath}
\begin{corr} \label{cor3.6}
There exist constants $c_{\mu\nu}^{-},c_{\mu\nu}^{+}$ and $C_{\mu\nu}^{-},C_{\mu\nu}^{+}$ depending only on the initial data and the universal constants such that all the entries of the Jacobian $\mathbf{J}$ and its inverse ${\mathbf{J}}^{-1}$ are uniformly bounded
\begin{displaymath}
c_{\mu\nu}^{-} \leq {\mathbf{J}}_{\mu\nu} \leq c_{\mu\nu}^{+}
\end{displaymath}
\begin{displaymath}
C_{\mu\nu}^{-} \leq {\mathbf{J}}_{\mu\nu}^{-1} \leq C_{\mu\nu}^{+}
\end{displaymath}
for $\mu,\nu=0,1,2$.
\end{corr}
\begin{proof}
The proof follows from Proposition 3.2 and Lemma 4.2.
\end{proof}
\begin{corr} \label{cor3.8}
For the scalar functions $r,R$, there exist constants $c_1,c_2$ such that the following pointwise estimates hold
\begin{displaymath}
r \geq c_1 R, \quad r \leq c_2 R.
\end{displaymath}
\end{corr}
\begin{proof}
By the $L^{\infty}$ estimate for the Jacobian and its inverse we have done in Corollary 4.4, we can obviously get the following estimates,
\begin{equation} \label{3.18}
|\partial_{R} r |=|\partial_v r-\partial_u r|=\frac{1}{2} |e^{\mathcal{F}-\beta}+e^{\mathcal{G}-\beta}| \leq c_1,
\end{equation}
\begin{equation} \label{3.19}
|\partial_{r} R |=\frac{1}{2}|\partial_r v-\partial_r u |=\frac{1}{2} |e^{-\mathcal{G}+\beta}+e^{-\mathcal{F}+\beta}| \leq c_2,
\end{equation}
The proof then follows by applying the fundamental theorem of calculus.
\end{proof}

With the use of Corollary 4.4, we can now construct a global null coordinate system, and we will consider our problem under this coordinate.

 \subsection{The Einstein equations in null coordinates}
In null coordinates, the components of the Einstein tensor take the following form
\begin{eqnarray*}
&&G_{00}=-e^{2\lambda}r^{-1} \partial_u(e^{-2\lambda}\partial_u r), \\
&&G_{01}=r^{-1}\partial_u\partial_v r, \\
&&G_{11}=-e^{2\lambda}r^{-1} \partial_v(e^{-2\lambda}\partial_v r), \\
&&G_{22}=-4r^2 e^{-2\lambda}\partial_u\partial_v \lambda.
\end{eqnarray*}
Other components are zero.

Thus, rewritting the Einstein-wave-Klein-Gordon system \eqref{2.10}-\eqref{2.15} in null coordinates, we can get
\begin{equation} \label{1.12}
\partial_u(e^{-2\lambda} \partial_u r)=-e^{-2\lambda} r(2{\gamma_u}^2+{\phi_u}^2)
\end{equation}
\begin{equation} \label{1.13}
r_{uv}=\frac{m^2}{4}r e^{2\lambda-2\gamma} {\phi}^2
\end{equation}
\begin{equation} \label{1.14}
\partial_v(e^{-2\lambda} \partial_v r)=-e^{-2\lambda} r(2{\gamma_v}^2+{\phi_v}^2)
\end{equation}
\begin{equation} \label{1.15}
\lambda_{uv}=-\gamma_u \gamma_v-\frac{1}{2} \phi_u \phi_v+\frac{m^2}{8} e^{2\lambda-2\gamma} {\phi}^2
\end{equation}
\begin{equation} \label{1.16}
\partial_u (r \partial_v \gamma)+\partial_v (r \partial_u \gamma)=\frac{m^2}{4}re^{2\lambda-2\gamma} {\phi}^2
\end{equation}
\begin{equation} \label{1.17}
\partial_u (r \partial_v \phi)+\partial_v (r \partial_u \phi)=-\frac{m^2}{2}r e^{2\lambda-2\gamma} \phi
\end{equation}
with stress-energy tensor,
\begin{eqnarray*}
&&T_{00}=2{\gamma_u}^2 +{\phi_u}^2, \\
&&T_{01}=\frac{m^2}{4} e^{2\lambda-2\gamma} {\phi}^2, \\
&&T_{11}=2{\gamma_v}^2 +{\phi_v}^2, \\
&&T_{22}=4r^2 e^{-2\lambda}\gamma_u \gamma_v+2r^2e^{-2\lambda} \phi_u \phi_v-r^2 e^{-2\gamma}\frac{m^2}{2}{\phi}^2.
\end{eqnarray*}

We can also rewrite the system in $(T,R)$ coordinates which reads
\begin{equation} \label{1.6}
\frac{r_T}{r} \lambda_T+\frac{r_R}{r} \lambda_R-\frac{r_{RR}}{r}={\gamma_T}^2+{\gamma_R}^2+\frac{1}{2}{\phi_T}^2+\frac{1}{2}{\phi_R}^2+e^{2\lambda-2\gamma} \frac{m^2}{2}{\phi}^2
\end{equation}
\begin{equation} \label{1.7}
-\frac{r_{TR}}{r}+\frac{r_T}{r} \lambda_R+\frac{r_R}{r} \lambda_T=2\gamma_T \gamma_R+\phi_T \phi_R
\end{equation}
\begin{equation} \label{1.8}
\frac{r_T}{r} \lambda_T+\frac{r_R}{r} \lambda_R-\frac{r_{TT}}{r}={\gamma_T}^2+{\gamma_R}^2+\frac{1}{2}{\phi_T}^2+\frac{1}{2}{\phi_R}^2-e^{2\lambda-2\gamma} \frac{m^2}{2}{\phi}^2
\end{equation}
\begin{equation} \label{1.9}
e^{-2\lambda} \lambda_{RR}-e^{-2\lambda} \lambda_{TT} =-e^{-2\gamma} \frac{m^2}{2}{\phi}^2+e^{-2\lambda}{\gamma_T}^2-e^{-2\lambda}{\gamma_R}^2+\frac{1}{2} e^{-2\lambda} {\phi_T}^2-\frac{1}{2} e^{-2\lambda} {\phi_R}^2
\end{equation}
\begin{eqnarray} \label{1.10}
\Box_{g}\gamma &=& -e^{-2\lambda}\gamma_{TT}+e^{-2\lambda} \gamma_{RR}-e^{-2\lambda}\frac{r_T}{r} \gamma_T+e^{-2\lambda} \frac{r_R}{r} \gamma_R \\ \nonumber
&=& -e^{-2\gamma}\frac{m^2}{2}{\phi}^2
\end{eqnarray}
\begin{eqnarray} \label{1.11}
\Box_{g} \phi &=& -e^{-2\lambda}\phi_{TT}+e^{-2\lambda} \phi_{RR}-e^{-2\lambda}\frac{r_T}{r} \phi_T+e^{-2\lambda} \frac{r_R}{r} \phi_R\\ \nonumber
&=& e^{-2\gamma} m^2 \phi
\end{eqnarray}
with the stress-energy tensor $T_{\mu\nu}$,
\begin{eqnarray*}
&&T_{00}={\gamma_T}^2+{\gamma_R}^2+\frac{1}{2}{\phi_T}^2+\frac{1}{2}  {\phi_R}^2+e^{2\lambda-2\gamma}\frac{m^2}{2}{\phi}^2=e^{2\lambda} \tilde{\mathbf{e}}, \\
&&T_{01}=2\gamma_T \gamma_R+ \phi_T \phi_R, \\
&&T_{11}={\gamma_T}^2+{\gamma_R}^2+\frac{1}{2}{\phi_T}^2+\frac{1}{2} {\phi_R}^2-e^{2\lambda-2\gamma}\frac{m^2}{2}{\phi}^2, \\
&&T_{22}=r^2 e^{-2\lambda}{\gamma_T}^2-r^2 e^{-2\lambda}{\gamma_R}^2+\frac{r^2}{2}e^{-2\lambda}{\phi_T}^2-\frac{r^2}{2}e^{-2\lambda} {\phi_R}^2-r^2 e^{-2\gamma}\frac{m^2}{2}{\phi}^2.
\end{eqnarray*}

Let $U=(\gamma,\phi)$, rewrite equations $(\ref{1.10})(\ref{1.11})$ in the following form on the Minkowski spacetimes $(\mathbb{R}^{1+2},m)$,
\begin{eqnarray} \label{1.18}
\Box_{m} U &=& -U_{TT}+U_{RR}+ \frac{1}{R} U_R\\ \nonumber
&=& e^{2\lambda}F(U){\phi}^2+e^{2\lambda}G(U)\phi+\frac{r_T}{r} U_T-\left(\frac{r_R}{r}-\frac{1}{R}\right) U_R \\ \nonumber
& \triangleq & h.
\end{eqnarray}
\subsection{Estimates on Energy Flux}
In view of the $L^{\infty}$ estimates of the Jacobian and the monotonicity of the energy , we give the following estimates on energy fluxes,
\begin{proposition} \label{prop3.11}
For $0 \leq T_0 \leq T_1$, consider the following space-time region
\begin{displaymath}
\{T_0\leq T \leq T_1\} \subset \mathcal{Q},
\end{displaymath}
where $\mathcal{Q}$ denotes the maximal development. There holds the following estimates
\begin{displaymath}
\int_{2T_0-v}^{\min{(v,2T_1-v)}} \left({(\partial_u \gamma)}^2+{(\partial_u \phi)}^2+{m^2}e^{-2\gamma}\phi^2 \right) (u',v) r du' \lesssim E(0),
\end{displaymath}
\begin{displaymath}
\int_{\max{(u,2T_0-u)}}^{T_1} \left({(\partial_v \gamma)}^2+{(\partial_v \phi)}^2+{m^2}e^{-2\gamma}\phi^2 \right) (u,v') r dv' \lesssim E(0).
\end{displaymath}
\end{proposition}
\begin{proof}
Noting that the vector field $P_T$ is divergence free, using Stokes' theorem, we can get
\begin{displaymath}
\left|  \int_{2T_0-v}^{\min{(v,2T_1-v)}} \int_{\theta=0}^{2\pi} e^{-\mathcal{F}} (\mathbf{e}-\mathbf{m}) {\bar{\mu}}_v \right| \leq E(0)
\end{displaymath}
and
\begin{displaymath}
\left|  \int_{\max {(u,2T_0-u)}}^{T_1} \int_{\theta=0}^{2\pi} e^{-\mathcal{G}} (\mathbf{e}+\mathbf{m}) {\bar{\mu}}_u \right| \leq E(0).
\end{displaymath}
where we use the notations and computations given in Section 4.2. In view of the definition of ${\bar{\mu}}_v$ and ${\bar{\mu}}_u$, and the rotation invariance, we deduce
\begin{displaymath}
 \int_{2T_0-v}^{\min{(v,2T_1-v)}}  e^{-\mathcal{F}} (\mathbf{e}-\mathbf{m}) r e^{2\lambda}du \lesssim E(0)
\end{displaymath}
and
\begin{displaymath}
 \int_{\max {(u,2T_0-u)}}^{T_1}  e^{-\mathcal{G}} (\mathbf{e}+\mathbf{m}) r e^{2\lambda}dv  \lesssim E(0).
\end{displaymath}
Then, by the definition of $\mathbf{e}$ and $\mathbf{m}$ and \eqref{3.12} \eqref{3.15}, we obtain
\begin{displaymath}
 \int_{2T_0-v}^{\min{(v,2T_1-v)}}  e^{\mathcal{G}} \left(e^{-2\mathcal{G}}{(\partial_u \gamma)}^2+e^{-2\mathcal{G}} {(\partial_u \phi)}^2+m^2 e^{-2\gamma} \phi^2\right) r du \lesssim E(0)
\end{displaymath}
and
\begin{displaymath}
 \int_{\max {(u,2T_0-u)}}^{T_1}  e^{\mathcal{F}} \left(e^{-2\mathcal{F}}{(\partial_v \gamma)}^2+e^{-2\mathcal{F}}{(\partial_v \phi)}^2+m^2 e^{-2\gamma} \phi^2 \right) r dv  \lesssim E(0).
\end{displaymath}
Therefore, using Lemma 4.2, we finally obtain the estimates on fluxes,
\begin{displaymath}
 \int_{2T_0-v}^{\min{(v,2T_1-v)}}   \left({(\partial_u \gamma)}^2+ {(\partial_v \phi)}^2+m^2 e^{-2\gamma} \phi^2\right) r du \lesssim E(0)
\end{displaymath}
and
\begin{displaymath}
 \int_{\max {(u,2T_0-u)}}^{T_1}  \left({(\partial_v \gamma)}^2+{(\partial_v \phi)}^2+m^2 e^{-2\gamma} \phi^2 \right) r dv  \lesssim E(0).
\end{displaymath}
\end{proof}

\subsection{Lower bound of $\gamma$}
In this part, we aim to show that $\gamma$ has a lower bound under the given initial conditions. We may achieve this by some translation. An easy observation tells that the 4-tuple $(\lambda+1,r,\gamma+1,\phi)$ also satisfy \eqref{1.6}-\eqref{1.11}. Then, we define
$$
\tilde{\lambda}=\lambda+1,\quad \tilde{\gamma}=\gamma+1.
$$

Now we will prove that $\tilde{\gamma}$ is nonnegative by a bootstrap argument. Firstly, we make a bootstrap assumption
\begin{equation} \label{3.24}
 r^{\frac{1}{2}}\tilde{\gamma} \geq 0.
\end{equation}

The goal will be to improve this bootstrap assumption.
\begin{proposition} \label{prop3.13}
Under the given initial conditions \eqref{1.19} and the bootstrap assumption \eqref{3.24}, for any $T>0,R>0$, there holds
 $$ r^{\frac{1}{2}}\tilde{\gamma} >0.$$
\end{proposition}
\begin{proof}
Multiplying \eqref{1.10} by $r^{\frac{1}{2}}$, we can verify that
\begin{displaymath}
{(r^{\frac{1}{2}} \tilde{\gamma})}_{TT}-{(r^{\frac{1}{2}} \tilde{\gamma})}_{RR}=e^{2\tilde{\lambda}-2\tilde{\gamma}} \frac{m^2}{2} r^{\frac{1}{2}} {\phi}^2-r^{-\frac{3}{2}} (r_u \cdot r_v) \tilde{\gamma}
\end{displaymath}
In view of the normalisation on the initial surface in 4.2, we see that $r^{\frac{1}{2}}\tilde{\gamma}$ is the solution of the above 1-dimensional wave equation with initial-boundary conditions
\begin{displaymath}
\begin{cases}
T=0: \, r^{\frac{1}{2}}\tilde{\gamma}=R^{\frac{1}{2}} (\gamma_0+1), {(r^{\frac{1}{2}} \tilde{\gamma})}_T=R^{\frac{1}{2}}e^{\beta-\alpha}\gamma_1, \\
R=0: \,  r^{\frac{1}{2}}\tilde{\gamma}=0.
\end{cases}
\end{displaymath}
where $\gamma_0,\gamma_1$ are the initial data of $\gamma$ and $\gamma_t$.

Then, when $R \geq T$, we have
\begin{eqnarray*}
(r^{\frac{1}{2}}\tilde{\gamma}) (T,R)&=&  \frac{1}{2} {(R+T)}^{\frac{1}{2}} (\gamma_0(R+T)+1) + \frac{1}{2} {(R-T)}^{\frac{1}{2}} (\gamma_0(R-T)+1)\\
&&+\frac{1}{2} \int_{R-T}^{R+T}{\xi}^{\frac{1}{2}} e^{\beta-\alpha}(0,\xi) \gamma_1(\xi) d \xi+\frac{1}{2} \int_0^T \int_{R-(T-\tau)}^{R+(T-\tau)} e^{2\tilde{\lambda}-2\tilde{\gamma}} \frac{m^2}{2} r^{\frac{1}{2}} {\phi}^2 d \xi d \tau \\
&&-\frac{1}{2}\int_0^T \int_{R-(T-\tau)}^{R+(T-\tau)}r^{-\frac{3}{2}} (r_u \cdot r_v) \tilde{\gamma} d \xi d\tau,
\end{eqnarray*}
and when $R<T$, we have
\begin{eqnarray*}
(r^{\frac{1}{2}}\tilde{\gamma}) (T,R)&=&  \frac{1}{2} {(R+T)}^{\frac{1}{2}} (\gamma_0(R+T)+1) - \frac{1}{2} {(T-R)}^{\frac{1}{2}} (\gamma_0(T-R)+1)\\
&&+\frac{1}{2} \int_{T-R}^{T+R} {\xi}^{\frac{1}{2}}e^{\beta-\alpha}(0,\xi) \gamma_1(\xi) d \xi+\frac{1}{2} \int_{0}^{T} \int_{|R-(T-\tau)|}^{R+(T-\tau)} e^{2\tilde{\lambda}-2\tilde{\gamma}} \frac{m^2}{2} r^{\frac{1}{2}} \phi^2 d \xi d \tau\\
&&-\frac{1}{2}\int_{0}^{T} \int_{|R-(T-\tau)|}^{R+(T-\tau)}r^{-\frac{3}{2}} (r_u \cdot r_v) \tilde{\gamma} d \xi d\tau.
\end{eqnarray*}
Firstly, we consider the lower bound of the case $R \geq T$. By Proposition 3.2, the initial conditions \eqref{1.19}, the bootstrap assumption \eqref{3.24}, the calculation of the Jacobian, and the equalities above, we have
\begin{eqnarray*}
r^{\frac{1}{2}}\tilde{\gamma} &\geq &  \frac{1}{2} {(R+T)}^{\frac{1}{2}} (\gamma_0(R+T)+1)+\frac{1}{2} {(R-T)}^{\frac{1}{2}}(\gamma_0(R-T)+1)\\
&&+\frac{1}{2} \int_{R-T}^{R+T}{\xi}^{\frac{1}{2}}e^{\beta-\alpha}(0,\xi) \gamma_1(\xi) d \xi+\frac{1}{2} \int_0^T \int_{R-(T-\tau)}^{R+(T-\tau)} e^{2\tilde{\lambda}-2\tilde{\gamma}} \frac{m^2}{2} r^{\frac{1}{2}} \phi^2 d \xi d \tau \\
&\geq & \frac{1}{2} {(R+T)}^{\frac{1}{2}} (\gamma_0(R+T)+1)+\frac{1}{2} {(R-T)}^{\frac{1}{2}}(\gamma_0(R-T)+1) \\
& >& 0
\end{eqnarray*}
Then, for the case $R<T$, we can estimate similarly,
$$
r^{\frac{1}{2}}\tilde{\gamma} \geq   \frac{1}{2} {(R+T)}^{\frac{1}{2}} (\gamma_0(R+T)+1) - \frac{1}{2} {(T-R)}^{\frac{1}{2}}(\gamma_0(T-R)+1)
$$
Noting that by \eqref{1.19}, on the initial surface, we have
\begin{eqnarray*}
\partial_r (r^{\frac{1}{2}}(\gamma+1)) &=& r^{\frac{1}{2}} \gamma_r+\frac{1}{2} r^{-\frac{1}{2}}(\gamma+1) \\
& \geq & r^{\frac{1}{2}} \gamma_r+\frac{1}{2} r^{-\frac{1}{2}} \\
&>& 0,
\end{eqnarray*}
if $r>0$. Therefore,
$$
r^{\frac{1}{2}}\tilde{\gamma} >0.
$$
when $R<T$.

Thus, we have shown that
\begin{displaymath}
 r^{\frac{1}{2}} \tilde{\gamma} >0.
\end{displaymath}
This finishes the proof of the proposition.
\end{proof}
The above proposition close the bootstrap argument and yields the nonnegativity of $\tilde{\gamma}$, which shows that $\gamma$ has a lower bound $\gamma \geq -1$.

\subsection{Existence of null coordinates in a light cone}
Now we consider a cone $J^{-}(O)$ with $O$ a point on the axis. As is done before, we can similarly construct a null coordinate system in $J^{-}(O)$ with different initial boundary conditions,
\begin{equation} \label{4.6.1}
 \begin{cases}
t=0: \tilde{u}=-\tilde{v}\\
r=0: \tilde{u}=t,\,\tilde{v}=t
 \end{cases}
 \end{equation}
The process of constructing such a coordinate system is just similar to what we have done in 4.1 and 4.2, and obviously the equations take the same form in this coordinate as in 4.3. The estimates we did before in section 4 still hold in the new null coordinate system.

This coordinate will be used only in the second part of our work of proving the global existence of the Einstein-wave-Klein-Gordon system. For simplicity, we still denote $(\tilde{u},\tilde{v})$ by $(u,v)$. Here, we give two more properties.
\begin{corr} \label{cor3.14}
There exist constants $c_1,c_2,c_3,c_4$ such that
\begin{displaymath}
r \geq c_1 R,\quad t \geq c_3 T,
\end{displaymath}
\begin{displaymath}
r \leq c_2 R, \quad t \leq c_4 T
\end{displaymath}
hold in $J^{-}(O)$.
\end{corr}
\begin{proof}
Similar to Corollary 4.5, using \eqref{3.18}\eqref{3.19}, we can obtain
\begin{equation*}
c_1 R \leq r \leq c_2 R.
\end{equation*}
Moreover, by the $L^{\infty}$ Estimate for the Jacobian, we have
\begin{equation}
|\partial_{T} t |=|\partial_v t+\partial_u t|=\frac{1}{2} |e^{\mathcal{F}-\alpha}+e^{\mathcal{G}-\alpha}| \leq {c'}_3,
\end{equation}
\begin{equation}
|\partial_{R} t|=|\partial_v t-\partial_u t|=\frac{1}{2} |e^{\mathcal{F}-\alpha}-e^{\mathcal{G}-\alpha}| \leq {c'}_4,
\end{equation}
\begin{equation}
|\partial_{t} T|=\frac{1}{2}|\partial_t v+\partial_t u|=\frac{1}{2} |e^{-\mathcal{F}+\alpha}+e^{-\mathcal{G}+\alpha}| \leq c_5,
\end{equation}
\begin{equation}
|\partial_{r} T|=\frac{1}{2}|\partial_r v+\partial_r u|=\frac{1}{2} |e^{-\mathcal{F}+\beta}-e^{-\mathcal{G}+\beta}| \leq c_6.
\end{equation}
Noting that in $J^{-}(O)$ there holds $R \leq T$, and $r \lesssim t$ owing to the fact that $\alpha$ and $\beta$ are bounded, then, as in Corollary 4.5, the proof follows by applying the fundamental theorem of calculus in the region $J^{-}(O)$.
\end{proof}
To work in the new coordinate system later, we need an energy estimate in $(T,R) $ coordinates.
We can similarly define energy in $(T,R)$ coordinates,
\begin{displaymath}
{\tilde{E}}^O(T):=\int_{{\tilde{\Sigma}_T} \cap {J^-(O)}} \tilde{\mathbf{e}} {\bar{\mu}}_{\tilde{q}}
\end{displaymath}
where $(\tilde{\Sigma}_T,\tilde{q})$ denotes the Cauchy surface. Then, we have
\begin{proposition} \label{prop3.15}
For $0 \leq T_0 \leq T_O$, we have that in $J^-(O)$, the following estimates hold
\begin{displaymath}
{\tilde{E}}^O(T_0) \lesssim E(0).
\end{displaymath}
\end{proposition}
\begin{proof}
Firstly, the proposition holds at the initial time. Then, for $T_0>0$, we will use Stokes' theorem for $P_T$ on the region
\begin{displaymath}
\tilde{K}(T_0,s):=\{T_0 \leq T, t  \leq s\}  \cap J^-(O).
\end{displaymath}
Using Stokes' theorem and the estimates on energy flux, we can get
\begin{equation} \label{4.6.2}
-\int_{{\tilde{\Sigma}_{T_0}} \cap {J^-(O)}} \langle P_T, \mathbf{n} \rangle {\bar{\mu}}_{\tilde{q}} \lesssim \varepsilon
\end{equation}
with $\mathbf{n}$ the unit normal vector,
\begin{displaymath}
\mathbf{n}=-e^{-\lambda} \partial_T=-e^{-\lambda}(\partial_u+\partial_v)
\end{displaymath}
Then, by Lemma 4.1, we have
\begin{displaymath}
\mathbf{n}=-\frac{(e^{-\lambda+\mathcal{F}-\alpha}+e^{-\lambda+\mathcal{G}-\alpha})}{2} \partial_t-\frac{(e^{-\lambda+\mathcal{F}-\beta}-e^{-\lambda+\mathcal{G}-\beta})}{2}\partial_r.
\end{displaymath}
Therefore,
\begin{displaymath}
-\langle P_T, \mathbf{n} \rangle=\frac{e^{-\lambda+\mathcal{F}}}{2}(\mathbf{e}+\mathbf{m})+\frac{e^{-\lambda+\mathcal{G}}}{2}(\mathbf{e}-\mathbf{m}).
\end{displaymath}
In view of the definition of $\mathbf{e}$ and $\mathbf{m}$ and \eqref{3.12} \eqref{3.15}, the $L^{\infty}$ Estimate for the Jacobian, Corollary 4.8, we imply
\begin{eqnarray*}
-\langle P_T, \mathbf{n} \rangle&=&\frac{e^{-\lambda+\mathcal{F}}}{2}(e^{-2\mathcal{F}}{(\partial_v \gamma)}^2+e^{-2\mathcal{F}}{(\partial_v \phi)}^2+m^2 e^{-2\gamma} \phi^2)\\
&&+\frac{e^{-\lambda+\mathcal{G}}}{2}(e^{-2\mathcal{G}}{(\partial_u \gamma)}^2+e^{-2\mathcal{G}}{(\partial_u \phi)}^2+m^2 e^{-2\gamma} \phi^2) \\
& \geq & C({(\partial_v \gamma)}^2+{(\partial_v \phi)}^2+{(\partial_u \gamma)}^2+{(\partial_u \phi)}^2+m^2 e^{-2\gamma} \phi^2) \\
& \geq & C \tilde{\mathbf{e}}.
\end{eqnarray*}
Thus, by \eqref{4.6.2}, we have
\begin{displaymath}
{\tilde{E}}^O(T_0) \lesssim \varepsilon.
\end{displaymath}
which concludes the proof of the proposition.
\end{proof}

\section{The first singularity occurs on the axis}
In this section, we prove that the first possible singularity must occur on the axis. This is equivalent to show that the Cauchy problem admits a regular solution away from the axis.
\begin{proposition} \label{prop4.1}
For any $(\bar{u},\bar{v})$ away from the axis, the solution of the system \eqref{1.6}-\eqref{1.11} is regular at $(\bar{u},\bar{v})$.
\end{proposition}
\begin{proof}
Away from the axis, the equations become a $1+1$ dimensional system. Now we denote the maximal development by $\mathcal{Q}$, without loss of generality, we consider $(\bar{u},\bar{v})\in \bar{\mathcal{Q}} \setminus \mathcal{Q}$. For $(\bar{u},\bar{v})\in \bar{\mathcal{Q}} \setminus \mathcal{Q}$, there exists a small light cone $\mathcal{C}_{(\bar{u},\bar{v})}$ away from the axis with vertex at $(\bar{u},\bar{v})$ satisfying
\begin{displaymath}
T_0 \leq T \leq T_1,\, 0<R_0 \leq R \leq R_1, \quad \forall (T,R) \in \mathcal{C}_{(\bar{u},\bar{v})},
\end{displaymath}
and
\begin{displaymath}
T_1-T_0<<1,\, R_1-R_0<<1.
\end{displaymath}
and
\begin{displaymath}
(\mathcal{C}_{(\bar{u},\bar{v})} \setminus \{ (\bar{u},\bar{v}) \}) \subset \mathcal{Q}.
\end{displaymath}
We denote the small scale of the light cone by $e$.

Now, by estimates on energy flux, we have
\begin{eqnarray*}
|U(\bar{u},\bar{v})| &\leq & |U(-\bar{v},\bar{v})|+ \int_{-\bar{v}}^{\bar{u}} |U_u| du \\
& \lesssim & C_0+{(\int_{-\bar{v}}^{\bar{u}} {|U_u|}^2 du)}^{\frac{1}{2}} \\
& \lesssim & C_0+{(\int_{-\bar{v}}^{\bar{u}} {|U_u|}^2 Rdu)}^{\frac{1}{2}} \\
& \leq & C(\bar{u},\bar{v}).
\end{eqnarray*}

Then,we rewrite the Einstein equations \eqref{1.16} and \eqref{1.17},
\begin{equation} \label{4.1}
\gamma_{uv}=-\frac{r_u}{2r} \gamma_v -\frac{r_v}{2r} \gamma_u+\frac{m^2}{8} e^{2\lambda-2\gamma} \phi^2,
\end{equation}
\begin{equation} \label{4.2}
\phi_{uv}=-\frac{r_u}{2r} \phi_v -\frac{r_v}{2r} \phi_u-\frac{m^2}{4} e^{2\lambda-2\gamma} \phi.
\end{equation}
Let
\begin{displaymath}
X=\sup\limits_{\mathcal{C}_{(\bar{u},\bar{v})}} |U_v|.
\end{displaymath}
Integrating \eqref{4.1} from $u_0$ to $u$, where $(u_0,v)$ on the initial hypersurface of $\mathcal{C}_{(\bar{u},\bar{v})}$ and $(u,v) \in \mathcal{C}_{(\bar{u},\bar{v})}$, we obtain
\begin{eqnarray*}
|\gamma_v| & \leq & M+X\int_{u_0}^u |\frac{r_u}{2r}| du +\int_{u_0}^u |\frac{r_v}{2r} \gamma_u| du+ \int_{u_0}^u \frac{m^2}{8} e^{2\lambda-2\gamma} \phi^2 du \\
& \lesssim & M+X\int_{u_0}^u du+\int_{u_0}^u | \gamma_u| du+ \varepsilon \\
& \lesssim & M+eX+{\left(\int_{u_0}^u {| \gamma_u|}^2 r du\right)}^{\frac{1}{2}} \\
& \lesssim & M+eX.
\end{eqnarray*}
Then we integrate \eqref{4.2} from $u_0$ to $u$, noting that $\gamma$ has a lower bound, we obtain
\begin{eqnarray*}
|\phi_v| & \leq & M+X\int_{u_0}^u |\frac{r_u}{2r}| du +\int_{u_0}^u |\frac{r_v}{2r} \phi_u| du+ \int_{u_0}^u \frac{m^2}{4} e^{2\lambda-2\gamma} \phi du \\
& \lesssim & M+X\int_{u_0}^u du+\int_{u_0}^u | \phi_u|+|e^{-2\gamma}\phi| du \\
& \lesssim & M+eX+{\left(\int_{u_0}^u ({| \phi_u|}^2+e^{-2\gamma} \phi^2) r du\right)}^{\frac{1}{2}} \\
& \lesssim & M+eX+E(0) \\
& \lesssim & M+eX.
\end{eqnarray*}
Thus,
\begin{displaymath}
X \lesssim M+eX
\end{displaymath}
which implies
\begin{displaymath}
X \lesssim M.
\end{displaymath}
We can similarly get upper bounds for $\partial_u U$. Therefore,
\begin{displaymath}
|\partial U| \lesssim M.
\end{displaymath}

Now we integrate \eqref{1.15} along the ingoing light ray, we get
\begin{eqnarray*}
|\lambda_v| & \lesssim & M+\int_{u_0}^u |\gamma_u||\gamma_v| du +\int_{u_0}^u |\phi_u||\phi_v| du+ \int_{u_0}^u \frac{m^2}{8} e^{2\lambda-2\gamma} \phi^2 du \\
& \lesssim & M+M{\left(\int_{u_0}^u ({| \gamma_u|}^2+{|\phi_u|}^2) r du\right)}^{\frac{1}{2}}+E(0) \\
& \lesssim & M.
\end{eqnarray*}
The same estimate holds for $\partial_u \lambda$. Then,
\begin{displaymath}
|\partial \lambda| \lesssim M.
\end{displaymath}
Thus, it must hold that $(\bar{u},\bar{v}) \in \mathcal{Q}$. This finished the proof of Proposition 5.1.
\end{proof}

\section*{Acknowledgement}
Both authors are grateful to Prof. Naqing Xie for fruitful discussions. The first author especially thanks him for his kind guidance.

Y. Zhou was supported by Key Laboratory of Mathematics for Nonlinear Sciences (Fudan University), Ministry of Education of China, P.R.China. Shanghai Key Laboratory for Contemporary Applied Mathematics, School of Mathematical Sciences, Fudan University, P.R. China, NSFC (grants No. 11421061, grants No.11726611, grants No. 11726612), 973 program (grant No. 2013CB834100) and 111 project.

\end{document}